\documentclass[12pt,reqno]{amsart}
\usepackage{graphicx}
\usepackage{amssymb}

\usepackage{hyperref}
\usepackage[all]{xy}
\usepackage{amssymb}

\usepackage{amsmath}

\usepackage{amsthm}
\usepackage{amscd}

\usepackage{amsfonts}

\usepackage{amssymb}

\usepackage{mathrsfs}

\usepackage{enumerate}

\usepackage{color,cite,graphicx}

\newcommand{\calD}{\mathcal{D}}

\newcommand{\calJ}{\mathcal{J}}

\newcommand{\sfR}{\mathsf{R}}

\newcommand{\frakr}{\mathfrak{r}}

\newcommand{\frakm}{\mathfrak{m}}

\newcommand{\frakS}{\mathfrak{S}}

\newcommand{\calL}{\mathcal{L}}

\newcommand{\calO}{\mathcal{O}}

\newcommand{\calR}{\mathcal{R}}

\newcommand{\be}{\mathcal{\begin{equation}}}
\newcommand{\ee}{\mathcal{\end{equation}}}

\newcommand{\bbE}{\mathbb{E}}
\newcommand{\bbG}{\mathbb{G}}

\newcommand{\bbT}{\mathbb{T}}

\newcommand{\bbF}{\mathbb{F}}
\newcommand{\bbC}{\mathbb{C}}

\newcommand{\bbP}{\mathbb{P}}
\newcommand{\bbQ}{\mathbb{Q}}
\newcommand{\bbR}{\mathbb{R}}

\newcommand{\bbZ}{\mathbb{Z}}

\newcommand{\sfr}{\textsf{r}}

\newcommand{\frakc}{\mathfrak{c}}

\newcommand{\Ref}{\textup{Ref}}

\newcommand{\Or}{\textup{O}}

\newcommand{\Spec}{\textup{Spec}~}
\newcommand{\Pic}{\textup{Pic}}

\newcommand{\Sp}{\textup{Sp}}

\newcommand{\id}{\textup{id}}

\newcommand{\Nm}{\textup{Nm}}

\newcommand{\NS}{\textup{NS}}

\newcommand{\Hom}{\textup{Hom}}

\newcommand{\rank}{\textup{rank}}

\newcommand{\la}{\langle}
\newcommand{\ra}{\rangle}
\newcommand{\half}{\tfrac{1}{2}}

\newcommand{\Aut}{\textup{Aut}}

\newcommand{\Num}{\textup{Num}}

\newcommand{\Tors}{\textup{Tors}}

\newcommand{\bmu}{\boldsymbol{\mu}}
\newcommand{\balpha}{\boldsymbol{\alpha}}
\newcommand{\bfpic}{\textbf{Pic}}
\newcommand{\Ker}{\text{Ker}}

\newtheorem{theorem}{Theorem}

\newtheorem{proposition}[theorem]{Proposition}

\theoremstyle{definition}     % italic or bold etc.
\newtheorem{definition}[theorem]{Definition}

\theoremstyle{remark}

\dedicatory{To Shigeru Mukai on the occasion of his 60th birthday}
\title[Enriques surfaces]{A brief introduction to Enriques surfaces}
\author{Igor V. Dolgachev}
\address{Department of Mathematics, University of Michigan, 525 E. University Av., Ann Arbor, Mi, 49109, USA}
\email{idolga@umich.edu}
\begin{document}
\begin{abstract} 
This is a brief introduction to the theory of Enriques surfaces over arbitrary algebraically closed fields.
\end{abstract}

\maketitle
\tableofcontents
\section{Introduction} This is a brief introduction to the theory of Enriques surfaces. Over $\bbC$, this theory can be viewed as a part of the theory of K3-surfaces, namely the theory of pairs $(X,\iota)$ consisting of a K3-surface $X$ and a fixed-point-free involution $\iota$ on $X$. It also can be viewed as the theory of lattice polarizes K3 surfaces, where the lattice $M$ is the lattice $U(2)\oplus E_8(2)$ with the standard notation of quadratic hyperbolic lattices \cite{DolgachevM}. The account of this theory can be found in many introductory lecture notes, for example, in \cite{Barth}, \cite{Kondo}, \cite{Mukai2}, and in books \cite{BarthBook} or \cite{BeauvilleBook}. We intentionally omit this theory and try to treat the theory of Enriques surfaces without appeal to their 
K3-covers. This makes more sense when we do not restrict ourselves with the basic field of complex numbers and take for the ground field an algebraically closed field of arbitrary characteristic $p \ge 0$.
This approach to Enriques surfaces follows the book of F. Cossec and the author \cite{CD}, the new revised, corrected and extended version of which is in preparation \cite{CDL}. 

The author shares his passion for Enriques surfaces with Shigeru Mukai and is happy to dedicate this survey to him.  He is grateful to the organizers of the conference, and especially to Shigeyuki Kond\={o}, for the invitation and opportunity 
to give a series of lectures on Enriques surfaces. He also thanks Daniel Allcock for providing proofs of some  group-theoretical  results  and the thorough referee for pointing out to many inaccuracies in 
its earlier version of the survey.

\section{History} Let $S$ be a smooth projective surface over an algebraically closed field $\Bbbk$. We use the customary notations from the theory of algebraic surfaces. Thus we reserve $D$ to denote a divisor on $S$ and very often identify it with the divisor class modulo linear equivalence. The group of such divisor classes is the Picard group $\Pic(S)$. The group of divisor classes with respect to numerical equivalence is denoted by $\Num(S)$. We denote by $|D|$ the linear system of effective divisors linearly equivalent to $D$. We set 
$$h^i(D) = \dim_\Bbbk H^i(S,\calO_S(D)), \ p_g = h^0(K_S) = h^2(\calO_S),\ q = h^1(\calO_S).$$
We use the Riemann-Roch Theorem 
$$h^0(D)-h^1(D)+h^2(D) = \half (D^2-D\cdot K_S)+1-q+p_g$$
and Serre's duality $h^i(D) = h^{2-i}(K_S-D).$

The theory of minimal models provides us with a birational morphism $f:S\to S'$ such that either the canonical class $K_{S'}$ is nef (i.e. $K_{S'}\cdot C\ge 0$ for any effective divisor $C$), or $S'$ is a projective bundle over $\Spec \Bbbk$, or over a smooth projective curve $B$.

If the latter happens the surface $S$ is called \emph{ruled} and, if $S' = \bbP^2$ or $B\cong \bbP^1$, it is called \emph{rational}. A rational surface 
has $p_g = q = 0$ since the latter are birational invariants. In 1894, Guido Castelnuovo tried to prove that the converse is true. He could not do it without an additional assumption that $h^0(2K_S) = 0$. He used  the so called  \emph{termination of  adjoints} (showing that, under this assumption $|C+mK_S| = \emptyset$ for  any curve $C$ and large $m$, and, if $m$ is minimal with this property, the linear system $|C+(m-1)K_S|$ gives a pencil of rational curves on $S$ that  implies the rationality). 

The modern theory of minimal models provides us with a simple proof of Castelnuovo's Theorem. First use that $D^2 \ge 0$ for any nef divisor $D$.\footnote{In fact, take any positive number $N$ and an ample divisor $A$, then $ND+A$ is ample and 
$(ND+A)^2 = N^2D^2+2NA\cdot D+A^2$ must be positive, however if $D^2 < 0$ and $N$ is large enough, we get a contradiction.}
Thus, if $S$ is not rational, then we may assume that $K_S$ is nef, hence $K_S^2 \ge 0$.  By Riemann-Roch,
$h^0(-K_S)+h^0(2K_S)\ge K_S^2+1$ implies $h^0(-K_S)\ge 1$, thus $-K_S\ge 0$ cannot be nef unless $K_S = 0$ in which case $p_g =1$. 

Still not satisfied, Castelnuovo tried to avoid the additional assumption that $h^0(2K_S) = 0$. He discussed this problem with  Enriques during their walks under arcades of Bologna. Each found an example of a surface with $p_g = q = 0$ with some effective multiple of $K_S$. Since the termination of the adjoint is a necessary condition for rationality, the surfaces are not rational.

The example of Enriques is a smooth normalization of a non-normal surface $X$ of degree 6 in $\bbP^3$ that passes with multiplicity 2 through the edges  of the coordinate tetrahedron. Its equation is 
$$F = x_1^2x_2^2x_3^2+x_0^2x_2^2x_3^2+x_0^2x_1^2x_3^2+x_0^2x_1^2x_2^2
+x_0x_1x_2x_3q(x_0,x_1,x_2,x_3) = 0,$$
where $q$ is a non-degenerate quadratic form. 

The surface $X$ has \emph{ordinary singularities}: a double curve $\Gamma$ with ordinary triple points that are also triple points of the surface, and a number of pinch points. The completion of a local ring at a general point is isomorphic to $\Bbbk[[t_1,t_2,t_3]]/(t_1t_2)$, at triple points $\Bbbk[[t_1,t_2,t_3]]/(t_1t_2t_3)$, and at pinch points $\Bbbk[[t_1,t_2,t_3]]/(t_1^2+t_2^2t_3)$. Let $\pi:S\to X$ be the normalization. The pre-image of a general point on $\Gamma$ consists of two points, the pre-image of a triple point consists of three points, and the pre-image of a pinch point consists of one point.

Let 
$\frakc_0 = \mathcal{H}om_{\calO_{X}}(\pi_*\calO_S,\calO_X)$. It is an ideal in $\calO_X$, called the \emph{conductor ideal}.  It is equal to the annihilator ideal of $\pi_*\calO_S/\calO_X$. Let  $\frakc = \frakc_0\calO_S$. This is an ideal in $\calO_S$ and $\pi_*(\frakc) = \frakc_0$. The duality theorem for finite morphisms gives an isomorphism
\begin{equation}
\label{can}
\omega_S = \frakc\otimes \pi^*\calO_{X}(d-4),
\end{equation}
where $\omega_S$ is the canonical sheaf on $S$ and $d = \deg X$ (in our case $d = 6$). In particular, it implies that 
$\frakc$ is an invertible sheaf isomorphic to $\calO_S(-\Delta)$, where $\Delta$ is an effective divisor on $S$. Under the assumption on singularities, $\frakc_0\cong \calJ_{\Gamma}$, hence $\Delta = \pi^{-1}(\Gamma)$.

Returning to our sextic surface, we find that $\deg \Gamma = 6$, the number $t$ of triple points is equal to $4$ and each edge contains 4 pinch points. The canonical class formula shows that $\omega_S \cong \pi^*\calO_S(2)(-\Delta)$. The projection formula gives $\pi_*\omega_S \cong \calO_X(2)\otimes \calJ_{\Gamma}$ (we use that $\frakc_0$ annihilates $\pi_*\calO_S/\calO_X$).
Since $\deg \Gamma = 6$, $\omega_S$ has no sections, i.e. $p_g(S) = 0$. Also, the exact sequence
$$0\to \calJ_\Gamma(2)\to \calO_{X}(2) \to \calO_\Gamma(2) \to 0$$
allows us to check that $H^1(S,\omega_S) \cong H^1(X,\calJ_\Gamma(2)) = 0$, i.e. $q = 0$. We use that the curve $\Gamma$ is an ACM-scheme, i.e. the canonical homomorphism of graded algebras 
$\oplus H^0(\bbP^3,\calO_{\bbP^3}(n))\to \oplus H^0(\Gamma,\calO_{\Gamma}(n))$ is surjective.

Now, 
$$\omega_S^{\otimes 2} \cong \calO_S(2K_S) \cong \frakc^{\otimes 2}\otimes \pi^*\calO_{\bbP^3}(2d-8)$$
$$ \cong \pi^*\calO_{X}(4)(-2\Delta) = \pi^*(\calO_X(4)\otimes \calJ_{\Gamma}^{<2>}),
$$
where $\calJ_{\Gamma}^{<2>}$ is the second symbolic power of the ideal sheaf $\frakc_0 = \calJ_\Gamma$, the sheaf of functions vanishing with order $\ge 2$ at a general point of $\Gamma$. The global section of the right-hand side defined by the union of four coordinate planes shows that $h^0(2K_S) > 0$, in fact, $\omega_S^{\otimes 2} \cong \calO_S$. 

It follows from the description of singularities of the sextic that the pre-image of each edge of the tetrahedron, i.e. an irreducible component of the double curve $\Gamma$, is an elliptic curve. The pre-image of the section of the surface with a face of the tetrahedron is the sum of three elliptic curves $F_1+F_2+F_3$, where $F_i\cdot F_j = 1, i\ne j$ and $F_i^2 =0$.
The pre-images of the opposite edges are two disjoint elliptic curves $F_i+F_{i}'$. The preimage of the pencil of quadrics with the base locus equal to the union of four edges excluding a pair of opposite edges is an elliptic pencil on $S$ of the form $|2F_i| = |2F_i'|$.

This example of Enriques was included in Castelnuovo's paper \cite{Castelnuovo} and was very briefly mentioned in Enriques foundational paper \cite{Enriques1}, n.39.  Enriques returned to his surface only much later, in a paper of 1906 \cite{Enriques2}, where he proved that any nonsingular surface with $q = p_g = 0, 2K_S \sim 0$ is birationally isomorphic to a sextic surface as above or its degeneration \cite{Enriques2}. Modern proofs of Enriques' results were given  in the sixties, in the dissertations of Boris Averbuch from Moscow \cite{Averbuch}, \cite{Averbuch2} and Michael Artin from Boston \cite{Artin}.

In his paper Castelnuovo considers  the birational transformation of $\bbP^3$ defined by the formula
\[T:(x_0:x_1:x_2:x_3) = (y_2y_3:y_0y_1:y_0y_2:y_0y_3).\]
Plugging in this formula in the equation of the sextic, we obtain
$$F(x_0,x_1,x_2,x_3) = y_0^4y_2^2y_3^2Q_1(y_0y_1,y_2y_3,y_1y_3,y_1y_2)$$
$$+y_0^3y_1y_2^2y_3^2Q_2(y_0y_1,y_2y_3,y_1y_3,y_1y_2).$$
After dividing by $y_0^3y_2^2y_3^2$, we obtain that the image of $X$ is a surface of degree 5 in $\bbP^3$ given by the equation
$$G = y_0Q_1(y_0y_1,y_2y_3,y_1y_3,y_1y_2)+y_1Q_2(y_0y_1,y_2y_3,y_1y_3,y_1y_2) = 0.$$
It has four singular points $[1,0,0,0], [0,1,0,0], [0,0,1,0], [0,0,0,1]$. The local computations  show that the first two points are double points locally isomorphic to a singularity $z^2+f_4(x,y) = 0$, where $f_4(x,y)$ is a binary form of degree four without multiple roots. Classics called such a surface singularity an \emph{ordinary tacnode} with  the \emph{tacnodal tangent plane}  $z = 0$. Nowadays we call such a singularity a \emph{simple elliptic singularity} of degree 2. Its minimal resolution has a smooth elliptic curve as the exceptional curve with self-intersection equal to $-2$. The other two singular points of the quintic surface are ordinary triple points (= simple elliptic singularities of degree 3).

One can show the converse: a minimal resolution of a normal quintic surface with  two ordinary triple points and two tacnodes with tacnodal tangent planes equal to faces of the tetrahedron with vertices at the singular points is an Enriques surface. In modern times, the quintic birational 
models of  Enriques surfaces were studied in  \cite{Kim}, \cite{Stagnaro}, \cite{Umezu}. 

 Consider the birational transformation of $\bbP^3$ given by the formula
\begin{equation}\label{stand}
(y_0:y_1:y_2:y_3) = (x_1x_2x_3:x_0x_2x_3:x_0x_1x_3:x_0x_1x_2).
\end{equation}
It transforms the sextic surface $V(F_6)$ to a birationally isomorphic sextic surface $V(G_6)$. The two birational morphisms $S\to \bbP^3$ are defined by linear systems $|H|$ and $|H+K_S|$.

Since $[0,0,0,1]$ is a triple point of the quintic surface $V(G)$, we can write its equation  in the form
$$G = x_3^2A_3(x_0,x_1,x_2)+2x_3B_4(x_0,x_1,x_2)+C_5(x_0,x_1,x_2) = 0,$$
Projecting from the triple point $[0,0,0,1]$, we get a rational double cover $V(G)\dasharrow \bbP^2$. Its branch curve is a curve of degree 8 given by the equation $B_4^2-C_5A_3 = 0.$ The projections of the tacnodal planes $y_0 = 0$ and $y_1 = 0$ are line components of this octic curve. The residual sextic curve has a double point at the intersection of these lines and two tacnodes with tacnodal tangent lines equal to the lines. This is an \emph{Enriques octic}.

%\begin{figure}[ht]
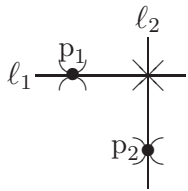
\xy
(-50,15)*{};(-50,-20)*{};
%(i);
(0,0)*{};(20,0)*{}**\dir{-};(15,5)*{};(15,-15)*{}**\dir{-};
(5,2)*\cir<5.2pt>{d_u};(5,-2)*\cir<5.2pt>{u_d};(17,-10)*\cir<5.2pt>{l_r};(13,-10)*\cir<5.2pt>{r_l};
(5,0)*{\bullet};(15,-10)*{\bullet};(-2,0)*{\ell_1};(15,7.5)*{\ell_2};
(13,-2)*{};(17,2)*{}**\dir{-};(17,-2)*{};(13,2)*{}**\dir{-};
(5,3)*{\txt\small{p}_1};(12,-10)*{\txt\small{p}_2};
\endxy
%\caption{}\label{branchcurve1}
%\end{figure}
In 1906 Enriques proved that any Enriques surface is birationally isomorphic to the double cover of $\bbP^2$ with branch curve as above or its degeneration \cite{Enriques2}.

Caslelnuovo  gave also his own  example of a non-rational surface with $q = p_g = 0$. It differs from Enriques' one by the property that $h^0(2K) = 2$. In this example, $S$ is given as a minimal resolution of a surface $X$ of degree 7 in $\bbP^3$ with the following singularities:
\begin{itemize}
\item a triple line $\ell$;
\item a double conic $C$ disjoint from $\ell$;
\item 3 tacnodes $p,q,r$ with tacnodal tangent planes $\alpha = 0,\beta = 0,\gamma = 0$ containing $\ell$.
\end{itemize}
The equation is 
$$F_7 = f_3^2h+\alpha\beta\gamma f_4 = 0,$$
where $h = 0$ is any plane containing the line $\ell$, $f_3 = 0$ is a cubic surface containing $\ell$ and $C$. The tacnodal planes are tangent planes to $f_3 =0 $ at the points $p,q,r$. The quartic surface $f_4 = 0$ contains $C$ as a double conic.

The pencil of planes through the line $\ell$ cuts out a pencil of quartic curves on $X$ with 2 nodes on $C$. Its members are birationally isomorphic to  elliptic curves. On the minimal resolution $S'$ of $X$, we obtain an elliptic fibration with a 2-section defined by the pre-image of the double conic. Each tacnodal tangent plane cuts out a double conic, and the pre-image of it on $S'$ is a divisor $2E_i+2F_i$, where $E_i$ is a $(-1)$-curve and $F_i$ is an elliptic curve. Blowing down $E_1,E_2,E_3$, we obtain a minimal elliptic surface $S$ with three double fibers. The canonical class is equal to $-F+F_1+F_2+F_3$. It is not effective. However, 
$2K_S \sim F$, so $h^0(2K) = 2$. 

\section{Generalities}
Recall  Noether's Formula
$$12(1-q+p_g) = K_S^2+c_2,$$
where $c_2 = \sum (-1)^ib_i(S)$ is the Euler characteristic in the usual topology if $\Bbbk = \bbC$ or $l$-adic topology otherwise.

In \emph{classical definition},   an Enriques surface is a smooth projective surface with $q=p_g = 0$ and  $2K_S = 0$. It is known that $q = h^1(\calO_S)$ is equal to the dimension of the tangent space of the Picard scheme $\bfpic_{S/\Bbbk}$. Thus its connected component $\bfpic_{S/\Bbbk}^0$ is trivial. The usual computation, based on the Kummer exact sequence, gives that $b_1 = 2\dim \bfpic_{S/\Bbbk}^0$. Thus $b_1 = 0$.  Noether's Formula   implies that $c_2 = 12$, hence  $b_2 = 10$. Also, since $2K_S = 0$,  $S$ is a minimal surface of Kodaira dimension $0$. A \emph{modern  definition} of an Enriques surface is the following (see \cite{BM}):

\begin{definition} An Enriques surface is a smooth projective minimal algebraic surface of Kodaira dimension 0 satisfying $b_1 = 0$ and $b_2= 10$. 
\end{definition}

 Other minimal surfaces of Kodaira dimension 0 are abelian surfaces with $b_1 = 4, b_2 = 6$, K3-surfaces with $b_1 = 0, b_2  = 22$, and hyperelliptic surfaces with $b_1 = b_2 = 2$ (see \cite{BM}). 
 
Let $S$ be an Enriques surface. Since the Kodaira dimension is zero, we obtain that  $K_S^2 = 0$. Also, since $h^0(K_S)$ is bounded, $p_g \le 1$.  Noether's Formula gives 
$q = p_g$. 

Recall that $\bfpic_{S/\Bbbk}^0$ parameterizes divisor classes algebraically equivalent to zero. It is an open and closed subscheme of the Picard scheme.  
$\bfpic_{S/\Bbbk}$ contains another closed and open subscheme $\bfpic_{S/\Bbbk}^\tau$ that  parameterizes divisor classes numerically equivalent to zero. The group $\bfpic_{S/\Bbbk}(\Bbbk)$ is the Picard group $\Pic(S)$ of divisor classes modulo linear equivalence. The group $\Pic^0(S): = \bfpic_{S/\Bbbk}^0(\Bbbk)$ is the subgroup of divisor classes algebraically equivalent to zero. The group $\Pic^\tau(S): = \bfpic_{S/\Bbbk}^\tau(\Bbbk)$ is the subgroup of numerically trivial divisor classes. The quotient group $\textup{NS}(S) = \Pic(S)/\Pic^0(S)$ is a finitely generated abelian group, the N\'eron-Severi group of $S$. The quotient group $\Pic(S)^\tau/\Pic^0(S)$ is the torsion subgroup $\Tors(\textup{NS}(S))$ of the N\'eron-Severi group and the quotient $\Pic(S)/\Pic^\tau(S)$ is isomorphic $\textup{NS}(S)/\Tors(\textup{NS}(S))$. It is a free abelian group denoted by $\Num(S)$.

If $p = 0$, all group schemes are reduced and $q = \dim \bfpic_{S/\Bbbk}^0$. In our case, this implies that $q = 0$. It is known that   $\bfpic_{S/\Bbbk}^0$ is reduced if $p_g = 0$ and, and for Enriques surfaces, this always happens if $p\ne 2$ \cite{BM2}. If $p = 2$ and $q = p_g =  1$, the group scheme 
$\bfpic_{S/\Bbbk}^0$ coincides with $\bfpic_{S/\Bbbk}^\tau$. It is a finite non-reduced group scheme of order 2 isomorphic to the group schemes $\bmu_2$ or $\balpha_2$. In the first case, an Enriques surface is called a $\bmu_2$-surface, and in the second case it is called an $\balpha_2$-surface, or \emph{supersingular surface} (because in this case the Frobenius acts trivially on $H^1(S,\calO_S)$ and $H^2(S,\calO_S)$).  

If $h^2(\calO_S) = h^1(\calO_S) = 0$, the Enriques surface $S$ is called \emph{classical}. In this case $\bfpic_{S/\Bbbk}^0 = 0$, $\Pic(S) = \textup{NS}(S)$ and   $\bfpic_{S/\Bbbk}^\tau$ is a constant group scheme defined by the group $\Tors(\textup{NS}(S))$. By Riemann-Roch, for any torsion divisor class $D\ne 0$ in $\Pic(S)$, we have 
$h^0(D)+h^0(K_S-D)\ge  1$. This implies that either $D$ or $K_S-D$ is effective. Since a non-trivial torsion divisor class cannot be effective, we have  $D\sim K_S$. 

It is known that $K_S$ is numerically trivial if the Kodaira dimension is equal to 0 (this is a highly non-trivial result, the core of the classification of algebraic surfaces). Since $h^0(2K_S)\ne 0$ because otherwise $S$ is rational, and $K_S$ is numerically trivial, $2K_S = 0$. So, $\Tors(\Pic(S)$ is of order $\le 2$. It is trivial if $q=p_g = 1$ and of order 2  otherwise.

If $p\ne 2$, the non-trivial 2-torsion element $K_S$ in $\Pic(S)$ gives rise to an \'etale double cover $f:X\to S$. We have 
$c_2(X) = 2c_2(S) = 24$, $K_X = f^*(K_S) = \calO_X$. Thus, $X$ is a K3-surface. If $p = 2$ and $S$ is a $\bmu_2$-surface, the same is true: there exists an \'etale double cover $f:X\to S$ and $X$ is a K3-surface. In other words, an Enriques surface in these cases is the quotient of a K3-surface by a fixed-point-free involution. So, the theory of Enriques surfaces becomes a chapter in the theory of K3 surfaces. This has been much overused in the modern literature by applying transcendental methods, in particular, the theory of periods of K3 surfaces, to solve some problems on Enriques surfaces of pure geometrical nature. These tools do not apply in the case of Enriques surfaces over fields of positive characteristic, however one can still cheat for some problems by lifting Enriques surfaces to characteristic 0.

For any finite commutative group scheme $G$ over $\Bbbk$, one has a natural isomorphism
$$\Hom_{\text{gr-sch}/\Bbbk}(D(G),\bfpic_{S/\Bbbk}) \cong H_{\text{fl}}^1(S,G),$$
where $D(G)$ is the Cartier dual of $G$ and the right-hand-side is the group of flat cohomology with coefficients in the sheaf represented by $G$. This group is isomorphic to the group of isomorphism classes of $G$-torsors over $S$. In our case, by taking $G = (\bbZ/2\bbZ)_\Bbbk, \bmu_2, \balpha_2$, we obtain $D(G) = \bmu_2,(\bbZ/2\bbZ)_\Bbbk, \balpha_2$, respectively. If $p\ne 2$ the groups $\bmu_2$ and $(\bbZ/2\bbZ)_\Bbbk$ are isomorphic. Hence, we have a non-trivial $(\bbZ/2\bbZ)_\Bbbk$-torsor if $p\ne 2$, or $S$ is a $\bmu_2$-surface. The corresponding degree 2 finite \'etale cover $\pi:X\to S$ is a K3-surface. The cover is known as the K3-cover of an Enriques surface. The Galois group of the cover is a group of order 2, acting freely on $X$ with the quotient isomorphic to $S$. Conversely, any such involution $\iota$ on a K3-surface, defines, after passing to the quotient map 
$X\to X/\la \iota\ra$ the K3-cover of the Enriques surface $S\cong X/\la\iota\ra$.

If $p = 2$ and $S$ is   a classical Enriques surface or an $\balpha_2$-surface, the non-trivial $\bmu_2$ or $\balpha_2$-torsor defines an inseparable degree 2 cover $X\to S$, also called the \emph{K3-cover}. However, the surface $X$ is not isomorphic to a K3-surface. It is birationally isomorphic to a K3-surface or it is a non-normal rational surface.

Let $\rho = \rank~\Pic(S) = \rank~\NS(S)$. If $\Bbbk = \bbC$, the Hodge decomposition $H^2(S,\bbC) = H^{2,0}+H^{1,1}+H^{0,2}$ implies that $H^{1,1} = b_2(S) = 10$. By the Lefschetz Theorem, all integral 2-cohomology classes are algebraic, hence $H^2(S,\bbZ) = \Pic(S)$ and $H^2(S,\bbZ)/\Tors = \Num(S) = \bbZ^{10}$. The Poincar\'e Duality implies that the intersection form on $\Num(S)$ is  a quadratic form on $\Num(S)$ defined by a symmetric matrix with determinant $\pm 1$. We say that $\Num(S)$ is a quadratic unimodular lattice. The adjunction formula $D^2 = 2\chi(\calO_S(D))-2$ for any irreducible effective divisor $D$ implies that $D^2$ is always even. The Hodge Index Theorem gives that the signature of the real quadratic space $\Num(S)\otimes \bbR$ is equal to $(1,9)$. Finally, Milnor's Theorem about even unimodular indefinite integral quadratic forms implies that $\Num(S) = U\perp E_8$, where 
$U$ is a hyperbolic plane over $\bbZ$ and $E_8$ is a certain negative definite unimodular even quadratic form of rank 8.

If $p \ne 0$, more subtle techniques, among them the duality theorems in \'etale and flat cohomology imply the same result provided one proves first that $\rho = b_2 = 10$. There are two proofs of this fact one by E. Bombieri and D. Mumford \cite{BM2} and another by W. Lang \cite{Lang}. The first proof uses the existence of an elliptic fibration on $S$, the second one uses the fact that an Enriques surface with no global regular vector fields can be lifted to characteristic 0. The fact  that  $\Num(S)$ is isomorphic to the lattice $U\oplus E_8$ was first proven by L. Illusie \cite{Illusie} who used  crystalline cohomology.  

One can use the following description of the lattice $U\oplus E_8$ which we denote by $\bbE_{10}$, sometimes it is called the \emph{Enriques lattice}. Let $\bbZ^{1,10}$ be the standard hyperbolic lattice with an orthonormal basis $e_0,e_1,\ldots,e_{10}$ satisfying $(e_i,e_j) = 0, (e_0,e_0) = 1, (e_i,e_i) = -1, i > 0$. Then $\bbE_{10}$ is isomorphic as a quadratic lattice to the orthogonal complement of the vector $k_{10} = 3e_0-\cdots-e_{10}$. The vectors 
$$\alpha_0 = e_0-e_1-e_2-e_3,\quad \alpha_i = e_i-e_{i+1},\  i = 1,\ldots,9$$
can be taken as a basis of $\bbE_{10}$. It is called a \emph{canonical root basis}. The  matrix of the symmetric bilinear form with respect to this basis is equal to $-2I_{10}+A$, where $A$ is the  incidence matrix of the graph pictured in Fig.2:

\begin{figure}[hp]
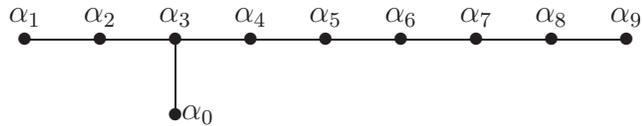

\xy (-30,0)*{};
@={(0,0),(10,0),(20,0),(30,0),(40,0),(50,0),(60,0),(70,0),(80,0),(20,-10)}@@{*{\bullet}};
(0,0)*{};(80,0)*{}**\dir{-};
(20,0)*{};(20,-10)*{}**\dir{-};
(0,3)*{\alpha_1};(10,3)*{\alpha_2};(20,3)*{\alpha_3};(30,3)*{\alpha_4};(40,3)*{\alpha_5};(50,3)*{\alpha_6};(60,3)*{\alpha_7};
(70,3)*{\alpha_8};(80,3)*{\alpha_9};(23,-10)*{\alpha_0};
\endxy
\caption{Enriques lattice}\label{enriqueslattice}
\end{figure}
The Enriques lattice $\bbE_{10}$ is isomorphic to the orthogonal complement of the canonical class of a rational surface obtained by blowing up 10 points in the projective plane. In fact, if we denote by $e_0$ the class of the pre-image of a line on the plane and by $e_i$ the classes of the exceptional divisors, we obtain that the canonical class is equal to $-3e_0+e_1+\cdots+e_{10}$, hence the claim. This explains the close relationship between the theory of Enriques surfaces and the theory of rational surfaces. In fact, if we take the 10 points in the special position, namely to be the double points of an irreducible rational curve of degree 6, the rational surface, called a \emph{Coble surface}, lies on the boundary of a partial compactification of the moduli space of Enriques surfaces.

We set 
$$f_i = e_i+k_{10},\ i = 1,\ldots, 10.$$
Since $(f_i,k_{10}) = 0$, these vectors belong to $\bbE_{10}$. We have 
$$(f_i,f_j) = 1-\delta_{ij},$$
where $\delta_{ij}$ is the Kronecker symbol. Also, adding up the expressions for the $f_i$'s, we obtain 
$$f_1+\cdots+f_{10} = 9k_{10}+3e_0 = 3(10e_0-3e_1-\cdots-3e_{10}).$$
We set
$$\delta = \frac{1}{3}(f_1+\cdots+f_{10}) = 10e_0-3e_1-\cdots-3e_{10}.$$
We have 
$$(\delta,\delta) = 10,\  (\delta,f_i) = 3,\ (f_i,f_j) = 1-\delta_{ij}.$$
A sequence of $k$ isotropic vectors in $\bbE_{10}$ satisfying the last property is called an \emph{isotropic $k$-sequence}. The maximal $k$ possible is equal to 10. An ordered isotropic 10-sequence defines a root basis in $\bbE_{10}$ as follows.   Consider the sublattice $L$ of $\bbE_{10}$ spanned by $f_1,\ldots,f_{10}$. The direct computation shows that its discriminant is equal to $-9$, thus it is a sublattice of index 3 in $\bbE_{10}$. The vector $\delta = \frac{1}{3}(f_1+\cdots+f_{10})$ has integer intersection with each $f_i$, hence it defines an element in the dual lattice $L^*$ such that $3\delta\in L$. This implies that $\delta\in \bbE_{10}$ and we may set 
\[
\alpha_0^* = \delta,\  \alpha_1^* = \delta-f_1,\ \alpha_2^* = 2\delta-f_1-f_2, \ \alpha_i^* = 3\delta-f_1-\cdots-f_i, i \ge 3.
\]
The vectors $(\alpha_1^*,\ldots,\alpha_{10}^*)$ form a basis of $\bbE_{10}$ and its dual basis $\alpha_0,\ldots,\alpha_9$ is a canonical root basis with the intersection graph as in Figure \ref{enriqueslattice}. 

The following matrix is the intersection matrix of the vectors $\alpha_i^*$. It was shown to me first by S. Mukai during our stay in Bonn in 1983. 

$$\left(\begin{array}{cccccccccc}10 & 7 & 14 & 21 & 18 & 15 & 12 & 9 & 6 & 3 \\
7 & 4 & 9 & 14 & 12 & 10 & 8 & 6 & 4 & 2 \\
14 & 9 & 18 & 28 & 24 & 20 & 16 & 12 & 8 & 4 \\
21 & 14 & 28 & 42 & 36 & 30 & 24 & 18 & 12 & 6 \\
18 & 12 & 24 & 36 & 30 & 25 & 20 & 15 & 10 & 5  \\
15 & 10 & 20 & 30 & 25 & 20 & 16& 12 & 8 & 4 \\
12 & 8 & 16 & 24 & 20 & 16 & 12 & 9 & 6 & 3 \\
9 & 6 & 12 & 18 & 15 & 12 & 9 & 6 & 4 & 2 \\
6 & 4 & 8 & 12 & 10 & 8 & 6 & 4 & 2 & 1 \\
3 & 2 & 4 & 6 & 5 & 4 & 3 & 2 & 1 & 0\end{array}\right).
$$

Let $\Or(\bbE_{10})$ be the orthogonal group of the lattice $\bbE_{10}$, i.e. the group of automorphisms of $\bbE_{10}$ preserving the quadratic form. We have 
$$\Or(\bbE_{10}) = W(\bbE_{10})\times \{\pm 1\},$$
where $W(\bbE_{10})$ is the Weyl group of $\bbE_{10}$ generated by reflections $s_{\alpha_i}:x\mapsto x+(x,\alpha_i)\alpha_i$. It coincides with the full \emph{reflection group} of $\bbE_{10}$, the group generated by $s_\alpha:x \mapsto x+(x,\alpha)\alpha$, where $\alpha$ is any vector with $(\alpha,\alpha) = -2$ (see, for example, \cite{DolgR}).

\section{Polarized Enriques surfaces}
The moduli space of Enriques surfaces exists as a stack only. If $p \ne 2$, it is an irreducible smooth unirational Artin stack of dimension 10. Over $\bbC$, it admits a coarse moduli space isomorphic to an arithmetic  quotient of a symmetric domain of orthogonal type (or type IV). If $p = 2$, it consists of two irreducible unirational components intersecting along a 9-dimensional substack. One component corresponds to classical Enriques surfaces and another one corresponds to $\bmu_2$-surfaces. The intersection corresponds to $\alpha_2$-surfaces. This is a recent result of Christian Liedtke \cite{Liedtke}.

To consider a quasi-projective moduli space one has to polarize the surface. A \emph{polarized surface} is a pair $(S,D)$, where $D$ is a nef divisor class with $D^2 > 0$ and $|D|$ is base-point-free. An isomorphism of polarized surfaces $(S,D)\to (S',D')$ is an isomorphism $f:S\to S'$ such that $f^*(D') \sim D$.  Let us discuss such divisor classes. 

Let $D$ be any irreducible curve on $S$. By adjunction formula, $D^2\ge -2$. If $D^2 = -2$, then $D\cong \bbP^1$. An Enriques surface containing a smooth rational curve is called \emph{nodal} and \emph{unnodal} otherwise. If $D^2\ge 0$,  by Riemann-Roch, $h^0(D)> 0$. Let $W_S^n$ be the subgroup of the reflection group of $\Num(S)$ generated by the reflections $s_\alpha$, where $\alpha$ is the class of a smooth rational curve. Applying elements of $W_S^n$, we obtain that 
$D\sim D_0+\sum R_i$, where $D_0$ is a nef effective divisor, and $R_i\cong \bbP^1$. If $D$ is nef and $D^2 \ge 2$, then $h^1(D) = 0$ and $\dim |D| = \half D^2$. If $D^2 = 0$, then $D = kE$, where $h^0(E) = 1$ but $h^0(2E) = 2$. The linear system $|2E|$ is a pencil of curves of arithmetic genus 1. It defines a morphism  $f:S\to \bbP^1$ whose general fiber is a curve of arithmetic genus 1. It is called a \emph{genus one fibration}. A general fiber is nonsingular if $p \ne 2$ but could have a cusp if $p = 2$. In the latter case the genus one fibration is called a \emph{quasi-elliptic fibration}.  If $K_S\ne 0$ (resp. $K_S = 0$), a genus one  fibration has two fibers (resp. one fiber)  of the form $2F$, called a \emph{double fiber}.

For any nef divisor $D$ with $D^2 > 0$, let $\Phi(D) = \min\{|D\cdot E|\}$, where $E^2 = 0$. The function $\Phi$ satisfies an inequality (see \cite{CD}, Corollary 2.7.1) 
$$\Phi(D)^2\le D^2.$$

We have $\Phi(D) = 1$ if and only if $|D|$ has base-points (two, counting with multiplicity). Also $\Phi(D) = 2$ if and only $|D|$ defines a double cover of a normal surface, or a birational morphism onto a non-normal surface, or $D^2 = 4$ and the map is of degree 4 onto $\bbP^2$. In the first case, the linear system is called \emph{superelliptic} (renamed to \emph{bielliptic} in \cite{CDL}). Finally, $\Phi(D)\ge 3$ if and only if $|D|$ defines a birational morphism onto a normal surface with at most rational double points as singularities.

Here are examples.

If $D^2 = 2$, then $D\sim E_1+E_2$ or $D\sim 2E_1+R+K_S$, where $|2E_i|$ are genus one pencils and $R\cong \bbP^1$ such that 
$E_i\cdot E_2 = 1$ and $R\cdot E_1 = 1$. The linear system $|D|$ is a pencil of curves of arithmetic genus 2.

Assume $D^2 = 4$ and $\Phi(D) = 1$, then, after blowing up the two base points, we obtain a degree 2 map to $\bbP^2$ with the branch divisor equal to an Enriques octic which may be degenerate  if $S$ is nodal. If $\Phi(D) = 2$,  and $S$ is unnodal, then $D\sim E_1+E_2$, where $|2E_i|$ are genus one pencils and $E_1\cdot E_2 = 2$. The map given by $|D|$ is a finite map of degree 4 onto $\bbP^2$. If $p\ne 2$, its branch locus is a curve of degree 12, the image of the dual of a nonsingular cubic curve under a map $\bbP^2\to \bbP^2$ given by conics \cite{Verra1}. If $p = 2$, the map could be inseparable. If $S$ is a $\bmu_2$-surface, the map is separable and its branch curve is a plane cubic. If $p\ne 2$ or  $S$ is a $\bmu_2$-surface, the preimage $\tilde{D}$ of $D$  on the K3-cover $X$ defines a linear system $|\tilde{D}|$ on $X$ that maps $X$ onto a complete intersection of three quadrics in $\bbP^5$.  
 
Assume $D^2 = 6$ and $\Phi(D) = 2$. Again, if $S$ is unnodal, then $D\sim E_1+E_2+E_3, $ where $|2E_i|$ are genus one pencils and $E_i\cdot E_j = 1$. The map is a birational map onto an Enriques sextic in $\bbP^3$. The moduli space of polarized surfaces $(S,D)$ admits a compactification,  a GIT-quotient of the space of sextic surfaces passing  through the edges of the tetrahedron with multiplicity two.  
 
 Assume $D^2 = 8$ and $\Phi(D) = 2$. If $S$ is unnodal, then $D \sim 2E_1+2E_2$ or $D\sim 2E_1+2E_2+K_S$, where $|2E_i|$ are genus one pencils and $E_1\cdot E_2 = 1$.  In the first case, the map given by the linear system $|D|$ is a double cover $\phi:S\to \calD_4$, where $\calD$ is a 4-nodal quartic del Pezzo surface. It is isomorphic to a complete intersection of two quadrics in $\bbP^4$ with equations
$$x_0x_1+x_2^2 = 0, \ x_3x_4+x_2^2 = 0.$$
Its minimal resolution is  isomorphic to the blow-up $X$ of five points in the projective plane equal to the singular points of an Enriques octic curve. The rational map $X\dasharrow \calD_4$ is given by the anti-canonical linear system. The cover ramifies over the singular points and a curve from $|\calO_\calD(2)|$.  Thus, birationally, the cover is isomorphic to the double cover of the plane branched over an Enriques octic. 

If $S$ is nodal, the degree 8 polarization can be also given by the linear system $|4E+2R|$ or $|4E+2R+K_S|$, where $|2E|$ is a genus one pencil and $R$ is a $(-2)$-curve with $E\cdot R = 1$. In the first case,  the linear system $|4E+2R|$ defines a degree 2 cover of a degenerate 4-nodal quartic del Pezzo surface. Its equations are 
$$x_0x_1+x_2^2 = 0, \ x_3x_4+x_0^2 = 0.$$
It has two ordinary nodes and one rational double point of type $A_3$. Its minimal resolution is isomorphic to the blow-up of four points in the plane equal to singular points of a degenerate Enriques octic.

 Figure 2 is the picture of the branch curve of the rational map 
$S\dasharrow \tilde{\calD}$, where $\tilde{\calD}$ is a minimal resolution of  singularities of $\calD$.  We assume here that  $p\ne 2$.

 \begin{figure}[ht]
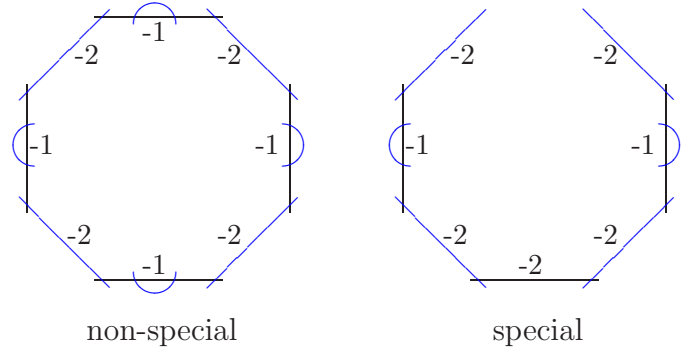

 \xy (-30,20)*{};
%verticar
(0,1)*{};(0,-16)*{}**\dir{-};(35,1)*{};(35,-16)*{}**\dir{-};
%horizontal
(9,10)*{};(26,10)*{}**\dir{-};(9,-25)*{};(26,-25)*{}**\dir{-};
%upper skew
(11,11)*{};(-1,-1)*{}**{\color{blue}\dir{-}};(24,11)*{};(36,-1)*{}**{\color{blue}\dir{-}};
%low skew
(-1,-14)*{};(11,-26)*{}**{\color{blue}\dir{-}};(36,-14)*{};(24,-26)*{}**{\color{blue}\dir{-}};
(1,-7)*{\color{blue}\cir<9pt>{l_r}};
(34,-7)*{\color{blue}\cir<9pt>{r_l}};
(17,9)*{\color{blue}\cir<9pt>{u_d}};
(17,-24)*{\color{blue}\cir<9pt>{d_u}};
(18,-32)*{\txt{non-special}};
(2,-7)*{\txt\small{-1}};(32,-7)*{\txt\small{-1}};(17,8)*{\txt\small{-1}};(17,-23)*{\txt\small{-1}};
(8,5)*{\txt\small{-2}};(27,5)*{\txt\small{-2}};(7,-19)*{\txt\small{-2}};(27,-19)*{\txt\small{-2}};
%special
%vertical
(50,1)*{};(50,-16)*{}**\dir{-};(85,1)*{};(85,-16)*{}**\dir{-};
%horizontal
(59,-25)*{};(76,-25)*{}**\dir{-};
%upper skew
(61,11)*{};(49,-1)*{}**{\color{blue}\dir{-}};(74,11)*{};(86,-1)*{}**{\color{blue}\dir{-}};
%low skew
(49,-14)*{};(61,-26)*{}**{\color{blue}\dir{-}};(86,-14)*{};(74,-26)*{}**{\color{blue}\dir{-}};
(51,-7)*{\color{blue}\cir<9pt>{l_r}};
(84,-7)*{\color{blue}\cir<9pt>{r_l}};
%(67,9)*{\color{blue}\cir<9pt>{u_d}};
(68,-32)*{\txt{special}};
(52,-7)*{\txt\small{-1}};(82,-7)*{\txt\small{-1}};
(58,5)*{\txt\small{-2}};(77,5)*{\txt\small{-2}};(57,-19)*{\txt\small{-2}};(77,-19)*{\txt\small{-2}};(67,-23)*{\txt\small{-2}};
\endxy
\caption{Branch curve of a bielliptic map ($p \ne 2$)}\label{branchcurve3}
\end{figure}
If $S$ admits a K3-cover, then the preimage of the linear system $|D|$ on the cover defines a degree 2 map onto $\bbP^1\times \bbP^1$ with branch curve of type $(4,4)$ invariant with respect to an involution of the quadric with four isolated fixed points. This is sometimes referred to as the \emph{Horikawa model}.

The  linear system $|D+K_S|$, where $|D| = |2E_1+2E_2|$ or $|4E+2R|$ as above,  maps $S$ birationally onto a non-normal surface of degree 8 in $\bbP^4$. So, we see that the type of polarization depends on the linear but not the numerical equivalence class of the divisor.

Note that all such linear systems exist on any Enriques surface, nodal or not.

Finally, assume $D^2 = 10$ and $\Phi(D) = 3$. If $S$ is unnodal, the linear system $|D|$ defines a closed embedding onto a surface $\textsf{F}$ of degree 10 in $\bbP^5$. Its homogeneous ideal is generated by 10 cubics. This model was first studied by Gino Fano in \cite{Fano2} and \cite{Fano1}. One can represent the divisor class $3D$ as the sum of 10 divisor classes $E_1+\cdots+E_{10}$ whose numerical classes form an isotropic 10-sequence in $\Num(S)$.  The images of $E_i$ and $E_i'\in |E_i+K_S|$ are plane cubics contained in $\textsf{F}$.
The linear system $ |D-E_i-E_j|, i\ne j,$ consists of an isolated genus one curve $E_{ij}$ which is mapped onto a curve of degree 4 on $\textsf{F}$. The linear system $|E_i+E_j+E_k|, k\ne i\ne j,$ maps $S$ onto an Enriques sextic $S'$ in $\bbP^3$,  the image of $E_{ij}$ is an elliptic quintic, and the images of $E_i$ and $E_j$ are coplanar edges of the tetrahedron. The images of the 7 curves $E_s, s\ne i,j,k,$ and 21 curves $E_{ab}, a,b\ne i,j,k,$ are plane cubic curves on $S'$. The residual cubic curve is 
the unique curve in the linear system $|E_1+E_2+E_3-E_s|$ or $|E_1+E_2+E_3-E_{ab}|$, they intersect at 6 points lying on the edges and three additional points. 
If $p\ne 2$, we also have the adjoint 28 curves $E_i'$ and $E_{ij}'$ numerically equivalent to $E_i$ and $E_{ij}$, respectively. The corresponding planes containing the images of the curves in a pair of adjoint curves intersect along a line  in a face of the tetrahedron.  
The image of the cubic curve $E_s$ (resp. $E_{ab}$) under the Cremona transformation \eqref{stand} is the cubic 
curve $E_s'$ (resp. $E_{ij}'$). 
 
Note that the numerical equivalence classes of the curves $E_i$ are determined uniquely by the choice of the Fano polarization $|D|$. A choice of an ordered representatives $E_i$ of these classes such that $3D\sim E_1+\cdots+E_{10}$ defines a \emph{supermarking} of $S$, i.e. a splitting of the projection $\Pic(S)\to \Num(S)$ preserving the intersection forms. A \emph{marking} of $S$ is just an isomorphism of quadratic lattices $\Num(S)\to \bbE_{10}$. So, there are $2^9$ supermarkings lifting a given marking. A supermarking of $S$ defines a choice of 10 planes in $\bbP^5$ cutting out 10 plane cubics on the Fano model $\textsf{F}$. One can show that  the moduli space of supermarked unnodal Enriques surfaces is irreducible and it is mapped into an irreducible component of the variety of ordered 10-tuples of mutually intersecting planes in $\bbP^5$ (see \cite{DM}).

If $S$ is nodal and $K_S \ne 0$, one of the  Fano polarization $D$ or $D+K_S$ maps $S$ into a nonsingular quadric in $\bbP^5$. If we identify the quadric with the Grassmann variety $G(2,4)$ of lines in $\bbP^3$, then the image of $S$ is isomorphic to the \emph{Reye congruence} of lines, the set of lines in a web of quadrics in $\bbP^3$ that are contained in a pencil from the web.  
Such  polarization of a nodal surface is called a Reye polarization \cite{CossecR}. If $D$ is a Reye polarization, then $|D+K_S|$ maps $S$ into $\bbP^5$ that can be identified with a general 5-dimensional linear system of quadrics. The image of $S$ is the locus of reducible quadrics.
 
 An interesting open problem is to determine the Kodaira dimension of the moduli space of polarized Enriques surface. If $D^2 = 4$ and $D = |E_1+E_2|$ with $E_1\cdot E_2 = 2$, then the moduli space is rational \cite{Casnati}.  
If $|D| = |E_1+E_2+E_3|$ is an Enriques sextic polarization, then, up to a projective transformation, a sextic model is defined uniquely by the quadratic form $q$. This shows that the moduli space is also rational. The moduli space of Enriques surfaces with polarization of degree 8 and type $|2E_1+2E_2|$ is birationally isomorphic to the GIT-quotient  
 $|\calO_{D_4}(2)|/\Aut(\calD_4)$. It can be shown to be rational \cite{CDL}. The moduli space of Enriques surfaces with a Fano polarization is birationally covered by the space of quintic elliptic curves in $\bbP^3$ \cite{Verra2}. It was shown in 
 loc.cit. that the latter space is rational and of dimension 10 . Thus the moduli space of Fano polarized Enriques surfaces is unirational. 
 
It is conjectured that the moduli space of polarized Enriques surfaces is always unirational (or, at least of negative Kodaira dimension).\footnote{A recent paper of V. Gritsenko and K. Hulek (\textit{Moduli of polarized Enriques surfaces}, math.AG.arXiv:1502.02723) disproves this conjecture.}  Note that, over $\bbC$, the coarse moduli space of Enriques surface is rational \cite{KondoR}.

\section{Nodal Enriques surfaces}

Recall that a nodal Enriques surface is an Enriques surface $S$ containing a smooth rational curve. By adjunction formula, the self-intersection of such curve is equal to $-2$, for this reason it is often called  a $(-2)$-curve. Over $\bbC$, a smooth rational curve $R$ on $S$ splits under the K3-cover $\pi:X\to S$ into the disjoint sum of two smooth rational curves $R_+$ and $R_-$. The Picard group $\Pic(X)$ contains the divisor class  $R_+-R_-$ that does not belong to $\pi^*(\Pic(S))$. The theory of periods of lattice polarized K3 surfaces shows that the nodal surfaces form an irreducible  hypersurface in the coarse moduli space of Enriques surfaces. Over any algebraically closed field of characteristic $p\ne 2$ one can show that a nodal Enriques surface is isomorphic to a Reye congruence of lines in $\bbP^3$. The moduli space of Reye congruences is an irreducible variety of dimension 9. On the other hand, the moduli space of  Enriques sextics is of dimension 10. This shows that a general Enriques surface is unnodal.

There are several invariants that measure how nodal an Enriques surface could be. The first one is the 
\emph{non-degeneracy invariant} $d(S)$ defined in \cite{CD}, p. 182. It is equal to the maximal $k$ such that there exists an isotropic $k$ sequence $(f_1,\ldots,f_k)$ in $\Num(S)$  where each $f_i$ is a nef numerical divisor class. If $S$ is unnodal, then $d(S) = 10$, maximal possible. It is known that $d(S) \ge 3$ if $p\ne 2$. However, no example of a surface with $d(S) = 3$ is known to me.  Note that this result implies, if 
$p\ne 2$,  that any Enriques surface admits a non-degenerate Enriques sextic model or a non-degenerate double octic model.

The next invariant was introduced by Viacheslav Nikulin \cite{Nikulin1}. To define it we assume that $p\ne 2$, or $S$ is a $\bmu_2$-surface. 

Let $\pi:X\to S$ be the K3-cover and  $\iota$ be the fixed-point involution with quotient isomorphic to $S$. Denote by 
$N^+$ ($N^-)$ the subgroup of $\Pic(X)$ that consists of invariant (anti-invariant) divisor classes. It is clear that 
$N^{-}$ is contained in the orthogonal complement $(N^+)^\perp$ in $\Pic(S)$. Also, since $G = \la \iota \ra$ acts freely, $N^+ = \pi^*(\Pic(S))$. Since $N^+$ contains an ample divisor, $N^-$ does not contain $(-2)$-curves. By the Hodge Index Theorem, $N^-$ is negative definite.  The quotient group 
$N^-/\textup{Im}(\iota^*-1)=\Ker(\iota^*+1)/\textup{Im}(\iota^*-1)$ is isomorphic to the  cohomology group $H^1(G,\Pic(X))$, where $G = \la \iota \ra$. The Hochshild-Serre spectral sequence in \'etale cohomology gives a boundary map $d_2:E_2^{1,1} = H^1(G,\Pic(X))\to E_2^{3,0} = H^3(G,\bbG_m(X)) \cong  \Hom(G,\textup{k}^*)$. Its kernel coincides with the kernel of the homomorphism of the Brauer groups $\pi^*:\textup{Br}(S)\to \textup{Br}(X)$, see \cite{Beauville}. It is shown in loc.cit. that $d_2$ coincides with the norm map $\Nm:\Pic(X)\to \Pic(S)$\footnote{Recall that the norm map is defined on invertible sheaves by setting $\Nm(\calL) = \det \pi_*\calL$.}
restricted to $\Ker(\iota^*+1)$ and its image is contained in $\Ker(\pi^*) = \langle K_S\rangle$. It is known that 
$\textup{Br}(X)$ is of order 2 if $K_S\ne 0$ and it is trivial otherwise (see \cite{CD}, Proposition 5.3.5). Thus, the order of $N^-/\textup{Im}(\iota^*-1)$ is at most 4, and in the case when $S$ is a $\bmu_2$-surface, the group is trivial.

Consider the subgroup $N_0^- = \textup{Im}(\iota^*-1)$ of $N^-$. For any $x\in \Pic(X)$, we have 
$\iota^*(x)+x\in \pi^*(\Pic(S))$, hence $(\iota^*(x)+x)^2 = 2x^2+2x\cdot \iota^*(x)\equiv 0\mod 4$, and we obtain that $x\cdot \iota^*(x)$ is even. This implies that $(x-\iota^*(x))^2\equiv 0\mod 4$. Thus the lattice 
$N_0^-(\half)$ is an integral even lattice. Note that $N^+(\half) \cong \bbE_{10}$.

Let us consider elements $\delta_- = (\iota^*-1)(x)$ in $N_0^-$ such that $\delta_-^2 = -4$ and $\delta_+:=(\iota^*+1)(x)$ is equal to $\pi^*(y)$, where $y^2 = -2$. Let $\Delta_S^+$ be the set of such classes $y$.
Note that, if  $\delta_- = (\iota^*-1)(x')$, then $x'= x+z$, where 
$\iota^*(z) = z$, hence $\delta_+= \iota^*(x)+x+2z = \pi^*(y')$, where  $y'=y+2z'$. Thus each $\delta_-$ determines a unique coset in  $\overline{\Num}(S)$. Let $\overline{\Delta}_S^+$ be the subset of such cosets. Let $W(\Delta_S^+)$ be the subgroup of $W_S$ generated by reflections in the classes of elements of $\Delta_S^+$. For any  $\alpha\in W(\Delta_S^+)$, we have $\pi^*(\alpha) = \iota^*(\beta)+\beta$ for some $\beta\in \Pic(S)$, so that, for  and $y\in \Delta_S^+$, we get 
$$\pi^*(s_\alpha(y)) = \pi^*(y+(y,\alpha)\alpha) = \iota^*(x+(y,\alpha)\beta)+(x+(y,\alpha)\beta).$$ 
This shows that $\Delta_S^+$ is invariant with respect to $W(\Delta_S^+)$.

Let
$$\overline{\Num}(S): = \Num(S)/2\Num(S) \cong \overline{\bbE}_{10}:= \bbE_{10}/2\bbE_{10} \cong \bbF_2^{10}.$$
We equip the vector space $\overline{\bbE}_{10}$ with the quadratic form $q:\overline{\bbE}_{10} \to \bbF_2$ defined by 
$$q(x+2\bbE_{10}) = \half  x^2\mod 2.$$
One can show that the quadratic form is non-degenerate and is of even type, i.e. equivalent to the orthogonal direct sum of five  hyperbolic planes $x_1x_2+x_3x_4+\cdots+x_9x_{10}$. Its orthogonal group is denoted by $\Or^+(10,\bbF_2)$. It contains a simple subgroup of index 2. We denote by $(x,y)$ the value of the associated symmetric form $b(x,y) = q(x+y)+q(x)+q(y)$ on a 
pair $x,y\in \overline{\Num}(S)$.

Let us identify $\overline{\Num}(S)$ with $\overline{\bbE}_{10}$. Nikulin defines the $r$-invariant of $S$ as the subset $\overline{\Delta_S}^+$ of $\overline{\bbE}_{10}$ equal to the image of $\Delta_S^+$ in $\overline{\Num}(S)$. 
This is a subset $J$ of $\bbF_2^{10}$ satisfying two properties
\begin{itemize}
\item $J \subset q^{-1}(1)$;
\item for any $r\in J, s_r(J)= J$, where $s_r(x) = x+(x,r)r$.
\end{itemize} 

Note that $\Delta_S^+$ contains divisor classes $y$ with $y^2 = -2$ such that $y\equiv R \mod 2\Num(S)$ for some $(-2)$-curve $R$ (it is easy to see, using Riemann-Roch on $X$,  that $y$ or $-y$ is effective). So, $\overline{\Delta_S}^+$ contains the set 
$\overline{\calR_S}$ of cosets of $(-2)$-curves on $S$ and the larger set $\overline{\calR_S}'$ obtained from this set by applying $s_r$, where 
$r\in \overline{\calR_S}$. I do not know whether this larger set coincides with $\overline{\Delta_S}^+$.

Another invariant, the $R$-invariant is defined by Nikulin as follows.  Let $K$ to be the sublattice of $N_0^-(\half)$ spanned by the classes $\delta_-$. It is a negative definite lattice spanned by vectors of norm $-2$. It follows that it is a root lattice, the  orthogonal sum of root lattices of types $A,D,E$. We also have a homomorphism $\phi:K/2K \to \overline{\Num}(S)$ that sends $\delta_-$ to $y+2\Num(S)$, where $\pi^*(y) = \delta_+$. Note that the image of $\phi$ is the linear subspace  spanned by $\overline{\Delta_S}^+$ and the image of the cosets of the $\delta_-$'s is the set $\overline{\Delta_S}^+$. The \emph{Nikulin $\sfR$-invariant} is the pair $(K,H)$, where $H = \Ker(\phi)$.  

A slightly different definition of the $R$-invariant was given by S. Mukai \cite{Mukai2}. He considers the kernel of the norm map $\Nm:\Pic(X)\to \Pic(S)$ and defines the \emph{root system} of $S$ as a sublattice of $\Ker(\Nm)(\half)$ generated by vectors with norm $(-2)$.

Over $\bbC$, one can use the theory of periods of K3 surfaces to show that Enriques surfaces with $\rank K = r$ form a codimension $r$ subvariety in the moduli space.

If $p = 2$ and $S$ is not a $\bmu_2$-surface, we still have the subset 
$\overline{\calR_S}'$ of $\bbF_2^{10}$ and we
define the \emph{r-invariant} of $S$ as the smallest subset $\sfr(S)$ of $\overline{\calR_S}'$ such that any $\overline{\calR_S}'$ can be written as a sum of elements from $\sfr(S)$. We picture $\sfr(S)$ as a graph with vertices in $\sfr(S)$ and the 
edges connecting two elements $x,y$ in $\sfr(S)$ such that $(x,y) = 1$.

\xy (0,5)*{};(0,-30)*{};
(0,0)*{\bullet};(-20,0)*{\#\sfr(S) = 1:};
(-20,-10)*{\#\sfr(S) = 2:};(0,-10)*{\bullet};(10,-10)*{\bullet};
(0,-10)*{};(10,-10)*{}**\dir{-};(30,-10)*{\bullet};(40,-10)*{\bullet};
(-6,-10)*{(a)};(23,-10)*{(b)};
(-20,-20)*{\#\sfr(S) = 3:};(0,-20)*{\bullet};(10,-20)*{\bullet};(5,-25)*{\bullet};
(0,-20)*{};(10,-20)*{}**\dir{-};(0,-20)*{};(5,-25)*{}**\dir{-};(10,-20)*{};(5,-25)*{}**\dir{-};
(-6,-20)*{(a)};(20,-20)*{(b)};(43,-20)*{(c)};(68,-20)*{(d)};
(25,-20)*{\bullet};(35,-20)*{\bullet};(30,-25)*{\bullet};
(25,-20)*{};(35,-20)*{}**\dir{-};
(50,-20)*{\bullet};(60,-20)*{\bullet};(55,-25)*{\bullet};
(75,-20)*{\bullet};(85,-20)*{\bullet};(80,-25)*{\bullet};
(75,-20)*{};(85,-20)*{}**\dir{-};(75,-20)*{};(80,-25)*{}**\dir{-};
\endxy

An Enriques surface is called a \emph{general nodal} if $\#\sfr(S) = 1$, i.e. any two $(-2)$-curves are congruent modulo $2\Num(S)$. In terms of the $\sfR$-invariant, it means that $(K,H) = (A_1,\{0\})$. The following theorem gives equivalent characterizations of  general nodal surfaces \cite{CDL}.

\begin{theorem}\label{T3.4}  The following properties are equivalent.
\begin{itemize}
\item[(i)] $S$ is a general nodal Enriques surface;
\item[(ii)]  Any genus one fibration on $S$ contains at most one  reducible fiber that consists of two irreducible components. A half-fiber is irreducible.
\item[(ii')] Any genus one fibration on $S$ contains at most one  reducible fiber that consists of two irreducible components.
\item[(iii)] Any two $(-2)$-curves are $f$-equivalent.
\item[(iv)]  For any  Fano polarization $h$, the set $\Pi_h = \{R\in \calR_S:R\cdot h\le 4\}$ consists of one element.
\item[(v)] For any $d\le 4$, $S$ admits a Fano polarization $h$ such that 
$\Pi_h = \{R\}$, where $R\cdot h = d$.
\item[(vi)] A genus one pencil that admits a smooth rational curve as a 2-section does not contain reducible fibres. 
\end{itemize} 
\end{theorem}
Here two $(-2)$-curves $R$ and $R'$ are called \emph{$f$-equivalent}\index{$f$-equivalence} if there exists a sequence of genus one  
fibrations $|2E_1|,\ldots,|2E_{k-1}|$ and a sequence of $(-2)$-curves $R_1 = R,\ldots,R_{k}= R'$ such that 
$$R_1+R_2\in |2E_1|, R_2+R_3\in |2E_2|,\ldots,  R_{k-1}+R_k\in |2E_{k-1}|.$$
Obviously, the $f$-equivalence is an equivalence relation on the set of nodal curves.

\section{Automorphisms of Enriques surfaces}

One of the main special features of Enriques surfaces is the richness of its symmetry group, i.e the group $\Aut(S)$ of birational automorphisms. Since $S$ is a minimal model, this group coincides with the group of biregular automorphisms. The  group of biregular automorphisms of any projective algebraic variety $X$ over $\Bbbk$ is the group of $\Bbbk$-points of a group scheme $\mathbf{Aut}_{X/\Bbbk}$ of locally finite type. This means that the connected component of the identity 
$\mathbf{Aut}_{X/\Bbbk}^0$ of $\mathbf{Aut}_{X/\Bbbk}$ is an algebraic group scheme over $\Bbbk$, and the group of connected components is  countable. The tangent space of $\mathbf{Aut}_{X/\Bbbk}^0$ at the identity point is isomorphic to the space of regular vector fields
$H^0(X,\Theta_{X/\Bbbk})$. All of this can be found, for example, in \cite{MO}.

It is known that in the case of an Enriques surface  $\dim H^0(S,\Theta_{S/\Bbbk})\le 1$ and the equality takes place only if $p = 2$ and $S$ is an $\balpha$-surface or an exceptional classical Enriques surface, the latter surfaces were described in \cite{Nick}. 

\begin{theorem} Let $S$ be an Enriques surface. Then  $\dim \mathbf{Aut}_{S/\Bbbk}^0 = 0$.  If  $h^0(\Theta_{S/\Bbbk}) = 0$,  then $\mathbf{Aut}_{S/\Bbbk}$ is reduced  and $\mathbf{Aut}_{S/\Bbbk}^{0}$ is trivial.  
\end{theorem}
 
\begin{proof} The second assertion follows immediately from the discussion in above.   Suppose $H^0(S,\Theta_{S/\Bbbk})\ne 0$. If $\mathbf{Aut}_{S/\Bbbk}$ is reduced then $\mathbf{Aut}_{S/\Bbbk}^0$  is a one-dimensional connected algebraic group $G$ over $\Bbbk$. There are three possibilities: $G = \bbG_m, \bbG_a$, or $G$ is an elliptic curve. A connected algebraic group acts trivially on the N\'eron-Severi group of $S$. Since $S$ has a non-trivial genus  one fibration with some rational fibers, the group $G$ preserves the set of singular fibers, and being connected, preserves any singular fiber. If $G$ is an elliptic curve, then $G$ must fix any point $x$ on a singular fiber. 
Thus $G$ acts linearly on any $\frakm_{S,x}^k/\frakm_{S,x}^{k+1}$, and being complete, it act trivially. This implies that $G$ acts trivially on the completion of the local ring $\calO_{S,x}$, hence on the ring itself, hence on its fraction ring, hence on $S$.

Suppose $G$ is a linear group of positive dimension acting on an irreducible algebraic variety $X$. We can always choose a one-dimensional subgroup of $G$ and assume that it acts freely on $X$. Then, by Rosenlicht's Theorem, there exists a $G$-invariant open subset $U$ of $X$ such that the geometric quotient $U\to U/G$ exists and its fibers are orbits of $G$. As is well-known this implies that $U$ is a principal homogenous space over $U/G$ (see, for example, \cite{Mumford}, Proposition 0.9). Thus $U$, and hence $X$, is  a ruled variety, i.e. it is birationally isomorphic to $\bbP^1\times U/G$.  Applying this to $S$, we find a contradiction.
\end{proof}

Enriques himself realized that $\Aut(S)$ is an infinite discrete group. In his paper \cite{Enriques2} of 1906 he remarks that any $S$ containing a general pencil of elliptic curves has infinite  automorphism group. The paper ends with the  question whether there exists a special degeneration of the sextic model such that the group of automorphisms is finite.

 A usual way to investigate $\Aut(S)$ is to consider its natural representation by automorphisms of some vector space or of 
 an abelian group. In our case, this would be $\Num(S)$. Since automorphisms preserve the intersection form, we have a homomorphism
$$\rho:\Aut(S) \to \Or(\Num(S)) \cong \Or(\bbE_{10}).$$
From now on, we fix an isomorphism $\Num(S)\cong \bbE_{10}$ and  identify these two lattices. Since automorphisms preserve the ample cone in $\Num(S)$, the image does not contain $-1_{\bbE_{10}}$, hence it is contained in the  preimage $W_S$ of the reflection group $W(\bbE_{10})$ in $\Or(\Num(S))$. It is a subgroup of $\Or(\Num(S))$ generated by reflections in numerical divisor classes $x$ with $x^2 = -2$.  The kernel of $\rho$ preserves  the set $\{h,h+K_S\}$, where $h$ is a very ample divisor 
class. Thus it preserves $2h$, hence it is contained  in a group of projective automorphisms of some projective space $\bbP^n$. It must be a linear algebraic group, hence  it is a finite group.

  The reduction homomorphism 
$\bbE_{10}\to \overline{\bbE}_{10} = \bbE_{10}/2\bbE_{10}$ defines a homomorphism of the groups
$W(\bbE_{10})\to \Or^+(10,\bbF_2)$. Let 
$$W(\bbE_{10})(2): = \Ker(W(\bbE_{10}) \to \Or^+(10,\bbF_2)).$$
It is called the \emph{2-level congruence subgroup} of $W(\bbE_{10})$. The preimage of $W(\bbE_{10})(2)$ in $W_S$ will be denoted by $W_S(2)$. It does not depend on a choice of an isomorphism $W_S\cong W(\bbE_{10})$.

\begin{proposition}[A. Coble] The subgroup $W(\bbE_{10})(2)$ is the smallest normal subgroup containing the involution 
$\sigma=1_U\oplus (-1)_{E_8}$ for some (hence any) orthogonal decomposition $\bbE_{10} = U\oplus E_8$.
\end{proposition}

\begin{proof} A proof suggested by Eduard Looijenga and following Coble's incomplete proof is  computational. It  is reproduced in \cite{CD}, Chapter 2, \S 10. The following nice short proof is due to Daniel Allcock.

Let $\Gamma$ be the minimal normal subgroup containing $\sigma$. It is generated by the conjugates of $\sigma$ in 
$W = W(\bbE_{10})$. Let $(f,g)$ be the standard basis of the hyperbolic plane $U$. If $\alpha_0,\ldots,\alpha_7$ is the basis of $E_8$ corresponding to the subdiagram of type $E_8$ of the Coxeter diagram of the Enriques lattice $\bbE_{10}$, then we may take 
$$f = 3\alpha_0+2\alpha_1+4\alpha_2+6\alpha_3+5\alpha_4+4\alpha_5+3\alpha_6+2\alpha_7+\alpha_8,$$
and $g = f+\alpha_9$. The stabilizer $W_f$ of $f$ in 
$W$ is the semi-direct product $E_8\rtimes W(E_8) \cong W(E_{9})$, where $E_9$ is the affine group of type $E_8$ and $W(E_9)$ is its Weyl group, the reflection group with the Coxeter diagram of type $T_{2,3,6}$. The image $\phi_v$ of 
$v\in E_8 = U^\perp$ under the map $\phi:E_8\to  W_f$ is the transformation 
$$\phi_v:x\mapsto x-(\frac{v^2}{2}(f,x)+(v,x))f+(x,f)v.$$
The inclusion of $W(E_8)$ in $W(E_9)$ is natural, it consists of compositions of the reflections in the  roots $\alpha_0,\ldots,\alpha_7$. In particular, any $w\in W(E_8)$ acts identically on $E_8^\perp = U$. The image of $\sigma$ in $W(E_9)$ is equal to 
$-\id_{E_8}\in W(E_8)\subset W(E_9)$. Let us compute the $\phi_v$-conjugates of $\sigma$. If $x\in E_8$, we have 
$$\phi_v\circ \sigma\circ \phi_{-v}(x)  = \phi_v(\sigma(x+(v\cdot x)f))$$
$$= 
\phi_v(-x+(v\cdot x)f) = (-x+(v\cdot x)f)+(v\cdot x)f = -x+2(v\cdot x)f.$$
Thus the intersection of $\Gamma$ with $W_f$ is equal to $\phi(2E_8):2$. The quotient $W_f/\Gamma\cap W_f$ injects into $\Or(\bar{\bbE}_{10}) \cong \Or^+(10,\bbF_2)$.  

Let us consider the subgroup $H$ generated by $W_f$ and $\Gamma$. Since $W_f$ normalizes $\Gamma\cap W_f$, the kernel of 
the homomorphism $H\to \Or(\bar{\bbE}_{10})$ coincides with $\Gamma$. To finish the proof it suffices to show that 
$H$ coincides with the preimage $W_{\bar{f}}$ in $W$ of the stabilizer subgroup of the image $\bar{f}$ of $f$ in $\Or(\bar{\bbE}_{10})$. Indeed, the kernel of $W_{\bar{f}} \to \Or(\bar{\bbE}_{10})$ is equal to $W(\bbE_{10})(2)$ and hence coincides with $\Gamma$.

Let us consider a sublattice $L$ of $\bbE_{10}$ generated by the roots $\alpha_0,\ldots,\alpha_8$, $\alpha_9'$, where 
$\alpha_9' = \alpha_8+2g-f$. The Dynkin diagram of this basis is the following

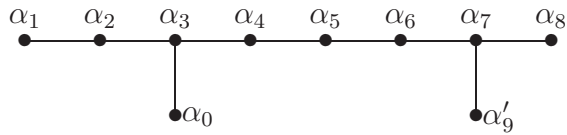
\xy (-30,10)*{};(-30,-15)*{};
@={(0,0),(10,0),(20,0),(30,0),(40,0),(50,0),(60,0),(70,0),(60,-10),(20,-10)}@@{*{\bullet}};
(0,0)*{};(70,0)*{}**\dir{-};(60,0)*{};(60,-10)*{}**\dir{-};
(20,0)*{};(20,-10)*{}**\dir{-};
(0,3)*{\alpha_1};(10,3)*{\alpha_2};(20,3)*{\alpha_3};(30,3)*{\alpha_4};(40,3)*{\alpha_5};(50,3)*{\alpha_6};(60,3)*{\alpha_7};
(70,3)*{\alpha_8};(63,-10)*{\alpha_9'};(23,-10)*{\alpha_0};
\endxy
Here all the roots, except $\alpha_9'$, are orthogonal to $f$. So, $H$ contains the reflections defined by these roots. 
Also the root $\alpha_8-f$ is orthogonal to $f$, and $\alpha_9'$ is transformed  to it under the conjugate of the group $\phi(2E_8)$ stabilizing $g$ (instead of $f$). So $H$ contains $s_{\alpha_9'}$ too. The Coxeter diagram contains three subdiagrams of  
affine types $\tilde{E}_8,\tilde{E}_8$ and $\tilde{D}_8$. The Weyl group is a crystallographic group with a Weyl chamber being a simplex of finite volume with 3 vertices at the boundary. This implies that $H$ has at most 3 orbits of ($\pm$ pairs) of primitive isotropic vectors in $\bbE_{10}$. On the other hand, $W_{\bar{f}}$ contains $H$ and has at least three orbits of them, because the stabilizer of $\bar{f}$ in $\Or(\bar{\bbE}_{10})$ has three orbits of isotropic vectors (namely, $\{\bar{f}\}$, the set of isotropic vectors distinct from $\bar{f}$ and orthogonal to $\bar{f}$, and the set of isotropic vectors not orthogonal to $\bar{f}$). This implies that the set of orbits of primitive isotropic vectors of $H$ and $W_{\bar{f}}$ is the same. Since the stabilizers of $f$ in these two groups are both equal to $W_f$, it follows that $H = W_{\bar{f}}$.

\end{proof}

\begin{theorem} Let $S$ be an unnodal Enriques surface. Then $\rho:\Aut(S)\to W_S$ is injective and the image  contains $W_S(2)$.
\end{theorem}

\begin{proof} In the complex case the assertion about the injectivity of the map follows from the classification of automorphisms that act identically on $\Num(S)$  due to S. Mukai and Y. Namikawa \cite{MN}, \cite{Muk}. Without assumption on the characteristic, one can deduce it from the arguments in  \cite{DolgachevT}. 

Consider the  linear system $|D| = |2E_1+2E_2|$ with $D^2 = 8$. As was explained in section 4, it defines a degree 2 map $S\to \calD_4\subset \bbP^4$, where $\calD_4$ is a 4-nodal  quartic del Pezzo surface. Let $\sigma$ be the deck transformation of the cover and $\sigma_* = \rho(\sigma)\in W_S$.\footnote{Since $S$ has no smooth rational curves, the cover is  a finite separable map of degree 2.} It is immediate that $\sigma_*$ leaves invariant the divisor classes of $E_1$ and $E_2$, and acts as the minus identity on the orthogonal complement of the sublattice $\la E_1,E_2\ra$ generated by $E_1,E_2$. The latter is isomorphic to the hyperbolic plane $U$ and $U^\perp$ is isomorphic to the lattice $E_8$ (sorry for the confusing notation). Now any conjugate of $\sigma_*$ in $W_S$ is also realized by some automorphism. In fact, $w\cdot \sigma_* \cdot w^{-1}$ leaves invariant $w(\la E_1,E_2\ra)$, and the deck transformation corresponding to the linear system $|2w(E_1)+2w(E_2)|$ realizes  $w\cdot \sigma_* \cdot w^{-1}$.

Now we invoke the previous proposition that says that  $W(\bbE_{10})(2)$ is the minimal normal subgroup of $W(\bbE_{10})$ containing $\sigma_*$.

\end{proof}

Here is the history of the theorem. We followed the proof of A. Coble in the similar case when $S$ is a Coble surface. This surface is a degeneration of an Enriques surface, although I do not know whether one can  deduce the result from Coble's theorem.

Over $\bbC$, the proof follows immediately from the \emph{Global  Torelli Theorem} for K3-surfaces. It was first stated by V. Nikulin \cite{Nik} and, independently, by  W. Barth and C. Peters \cite{BP}. Also the Global Torelli Theorem implies that the subgroup generated by $\Aut(S)$ and the subgroup $W_S^n$ generated by reflections in the classes of $(-2)$-curves is of finite index in $W_S$. In particular, $\Aut(S)$ is finite if and only $W_S^n$ is of finite index in $W_S$. If $S$ has no nonzero regular vector fields, then the same is true in any characteristic. This follows easily from the proof of Theorem 2.1 in a recent paper \cite{Maulik}. 

If $\Bbbk = \bbC$, one can show that for a general, in the sense of moduli, Enriques surface, $\Aut(S)\cong W(\bbE_{10})(2)$. 

 The interesting case is when $S$ is a nodal Enriques surface. Over $\bbC$, Nikulin proves that, up to finite groups, $\Aut(S)$ is determined by the $r$ or the $R$-invariant of $S$. He deduces the following theorem from the Global Torelli Theorem for K3-surfaces \cite{Nikulin1}.

\begin{theorem} [V. Nikulin]\label{nik2} Let $W(S;\Delta_S^+)$ be the subgroup of  $W_S$ leaving invariant the set 
$\Delta_S^+$, and let   $W(\Delta_S^+)$ be the  normal subgroup of $W(S;\Delta_S^+)$ generated by reflections in such curves.  Then the homomorphism 
$\rho:\Aut(S) \to W(S;\Delta_S^+)/W(\Delta_S^+)$ has a finite kernel and a finite cokernel. 
\end{theorem}

 Let $\sfr(S)$ be the $r$-invariant of $S$ and let  $\la \sfr(S)\ra$ be the subspace of $A_S$ generated by $\sfr(S)$ and $\sfR_S$ be the preimage of $\la \sfr(S)\ra^\perp$ under the reduction modulo $2\bbE$ map. This is called the \emph{Reye lattice} of $S$. An equivalent definition is 
$$\sfR_S = \{x\in \Num(S): x\cdot R \equiv 0 \mod 2\  \textrm{for any $(-2)$-curve}\ R\}.$$
Obviously, the action of $\Aut(S)$ on $S$ preserves the set of $(-2)$-curves, hence preserves the Reye lattice. Thus we have a homomorphism
$$\rho:\Aut(S) \to \Or(\sfR_S).$$
 Let $A_{\sfR_S} = \sfR_S^\vee/\sfR_S$ be the discriminant group of $\sfR_S$. Since any element in the image of $\rho$ lifts to an isometry of the whole lattice $\bbE_{10}$, it must be contained in the kernel of the natural homomorphism 
 $\Or(\sfR_S)\to \Or(A_{\sfR_S})$.

Suppose $S$ is a general nodal Enriques surface. In this case the Reye lattice $\sfR_S$ is  a sublattice of $\Num(S)$ of index 2. It is isomorphic to the lattice $U\oplus E_7\oplus A_1$. One can choose a basis defined by $\delta, E_1,\ldots,E_{10}$ as above such that the nontrivial coset is equal to the coset of the divisor class $\alpha = \delta-2E_{10}$. Then $\sfR_S = \{x\in \bbE_{10}:x\cdot \alpha\in 2\bbZ\}$ has a basis formed by the divisor classes
$$\beta_0 = \delta-E_1-\cdots-E_4,\ \beta_i = E_i-E_{i+1}, \ i = 1,\ldots,9.$$
The matrix of the quadratic lattice is equal to $-2I_{10}+B$, where $B$ is the incidence matrix of the graph:

\begin{figure}[ht]
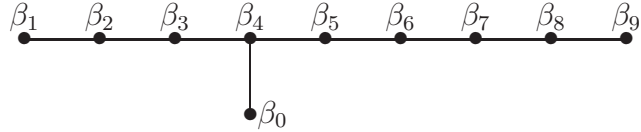

\xy (-30,0)*{};
@={(-10,0),(0,0),(10,0),(20,0),(30,0),(40,0),(50,0),(60,0),(70,0),(20,-10)}@@{*{\bullet}};
(-10,0)*{};(70,0)*{}**\dir{-};
(20,0)*{};(20,-10)*{}**\dir{-};
(0,3)*{\beta_2};(10,3)*{\beta_3};(20,3)*{\beta_4};(30,3)*{\beta_5};(40,3)*{\beta_6};(50,3)*{\beta_7};(60,3)*{\beta_8};
(70,3)*{\beta_9};(-10,3)*{\beta_1};(23,-10)*{\beta_0};
\endxy
\caption{Reye lattice}\label{}
\end{figure}
We denote this quadratic lattice by $E_{2,4,6}$ and call it the \emph{Reye lattice}.

Note that the Reye lattice is 2-reflective, i.e. the subgroup $\Ref_{2}(E_{2,4,6})$ generated by reflections in  vectors $\alpha$ with $\alpha^2 = -2$ is of finite index in the orthogonal group (see \cite{DolgR}, Example 4.11). However, it is larger than the Weyl group $W_{2,4,6}$ of the lattice $E_{2,4,6}$ generated by reflections in the vectors $\beta_i$. The former group is generated by the reflections $s_{\beta_i}$ and the vector $\frakr = f-\beta_9$, where $f = 2\beta_0+\beta_1+2\beta_2+3\beta_3+4\beta_4+3\beta_5+2\beta_6+\beta_7$ is an isotropic vector. We have $\frakr\cdot \beta_9 = 2$ and $\frakr\cdot \beta_i = 0, i\ne 9$.  The Coxeter graph of the full 2-reflection group of $E_{2,4,6}$ is the following:

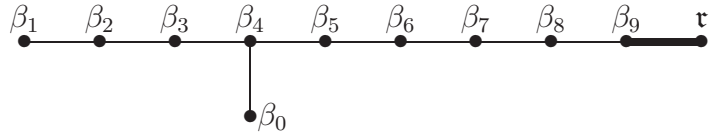
\xy (-30,10)*{};(-30,-15)*{};
@={(-10,0),(0,0),(10,0),(20,0),(30,0),(40,0),(50,0),(60,0),(70,0),(80,0),(20,-10)}@@{*{\bullet}};
(-10,0)*{};(70,0)*{}**\dir{-};
(20,0)*{};(20,-10)*{}**\dir{-};(70,0.1)*{};(80,0.1)*{}**\dir{-};(70,0)*{};(80,0)*{}**\dir2{-};(70,-0.1)*{};(80,-0.1)*{}**\dir{-};
(0,3)*{\beta_2};(10,3)*{\beta_3};(20,3)*{\beta_4};(30,3)*{\beta_5};(40,3)*{\beta_6};(50,3)*{\beta_7};(60,3)*{\beta_8};
(70,3)*{\beta_9};(-10,3)*{\beta_1};(23,-10)*{\beta_0};(80,3)*{\frakr};
\endxy

Since $\Ref_2(E_{2,4,6})$ is of finite index in $\Or(E_{2,4,6})$, and since $E_{2,4,6}$ is of finite index in $\bbE_{10}$, it is of finite index in $W(\bbE_{10})$. In particular, the 2-level congruence subgroup of $\Ref_2(E_{2,4,6})(2)$ is of finite index in $W(\bbE_{10})$. If we choose an isomorphism  $\sfR_S \to E_{2,4,6}$ such that $\frakr$ represents the class of a $(-2)$-curve, then the image of $W_S^n$ is a normal subgroup of $\Ref_2(E_{2,4,6})(2)$ and the quotient is isomorphic to the 2-level subgroup of $W_{2,4,6}$. 

If $S$ admits a K3-cover, then Nikulin's $r$-invariant $\overline{\Delta_S}^+$ consists of one element. The set $\Delta_S^+$ consists of vectors of norm $-2$ in the unique coset in $\overline{\Num(S)}$ that contains a $(-2)$-curve. If $\Bbbk = \bbC$, 
applying Nikulin's Theorem \ref{nik2}, we obtain, that, up to a finite group, $\Aut(S)$ is isomorphic to 
$W_{2,4,6}$. 

The next theorem gives a much more precise result about the structure of the group of automorphisms of a general nodal Enriques surface.  

Since $\sfR_S$ has discriminant group isomorphic to $(\bbZ/2\bbZ)^2$, its reduction modulo $2\sfR_S$ is a  quadratic vector space. with 2-dimensional radical. The radical is generated by the cosets of $\frakr$ and $\beta_0+\beta_1+\beta_3$.  and its orthogonal group is isomorphic to $G = 2^8\rtimes \Sp(8,\bbF_2)$.  It is known that the homomorphism $W(\sfR_S)\to G$ is surjective \cite{Coble1},\cite{Griess}. Let $W(\sfR_S)(2)'$ be equal to the pre-image of the subgroup $2^8$. Obviously it contains the 2-level congruence subgroup $W(\sfR_S)(2)$ as a subgroup of index $2^8$.

\begin{theorem} Let $S$ be a general nodal Enriques surface and $\rho:\Aut(S)\to W(\sfR_S)$ be its natural representation. Then the kernel of $\rho$ is trivial, and the image is equal to $W(\sfR_S)(2)'$.
\end{theorem}

\begin{proof} Consider the following sublattices of $E_{2,4,6}$.
\begin{itemize}
\item $L_1 =\la \beta_0,\ldots,\beta_6\ra$. It isomorphic to $E_7$ and $L_1^\perp \cong A_1\oplus U$.
\item $L_2 = \la \beta_0,\beta_2,\ldots,\beta_7\ra$. It is isomorphic to $E_7$ and $L_2^\perp\cong A_1\oplus U(2)$.
\item $L_3 = \la \beta_0,\beta_2,\ldots,\beta_8\ra$. It is isomorphic to $E_8$ and $L_3^\perp\cong A_1\oplus A_1(-1)$.
\end{itemize}
Define the following involutions of $\bbT_{2,4,6}$:
\begin{itemize}
\item $K = -\id_{L_1}\oplus \id_{L_1^\perp}$;
\item $G = (-\id_{L_2}\oplus \id_{L_2^\perp})\circ s_{\beta_9}$;
\item $B = -\id_{L_3}\oplus \id_{L_3^\perp}$;
\end{itemize}

One  proves in \cite{CD}, again following A. Coble, that the minimal normal subgroup $\la \la B,K\ra \ra$ containing $B,K$ coincides with $W(\sfR_S)(2)$ and together with $G$, the three involutions $K,B,G$ normally generate $W(\sfR_S)(2)'$.\footnote{The  computational proof from \cite{CD} can now be replaced by a conceptual proof due to D. Allcock (\textit{Congruence subgroups and Enriques surface automorphisms}, math.AG.arXiv:1601.00103).} It remains to show that all of these involutions and their conjugate can be realized by automorphisms of $S$.

The involution  $K$ is realized by the deck transformation of the double cover $S\to C_3$ of a cubic surface $C_3$ defined by a linear system $|2E_1+2E_2-R|$, where $E_1,E_2$ are half-fibers with $E_1\cdot E_2 = 1$ and $R$ is a smooth rational curve with $R\cdot E_1 = R\cdot E_2 = 0$. One can show that it always exists. 

The involution $B$ is realized by the deck transformation of the double cover $S\to \calD'$ onto a degenerate 4-nodal quartic del Pezzo surface defined by the linear system $|4E_1+2R|$, where $E_1$ is a half-fiber and $R$ is a smooth rational curve with $E_1\cdot R = 1$.

Finally, the involution $G$ is realized by the double cover $S\to \calD$ defined by the linear system $|2E_1+2E_2|$ which we considered in the previous discussion. Note that $G\in W(\bbE_{10})(2)$ but does not belong to $W(\sfR_S)(2)$.
\end{proof}

The letters B, G and K  here stand for E. Bertini, C. Geiser and S. Kantor. The K3-cover $\pi:X\to S$ of a general nodal Enriques surface is birationally isomorphic to a quartic symmetroid $Y$. This is a quartic surface in $\bbP^3$ with 10 nodes, its equation is given by the determinant of a symmetric matrix with entries linear forms in the projective coordinates. The surface $X$ admits a birational map $\sigma:X\to Y$, so it is a minimal resolution of $Y$. Let $Q_1,\ldots,Q_{10}$ be the exceptional curves and $H$ is the class of a pre-image of a plane section of $Y$.  For a general $X$, the Picard group of $X$ is generated by $H,Q_i$ and $H'$ such that $2H'\sim 3H-Q_1-\cdots-Q_{10}$. The orthogonal complement of the divisor class of $2H-Q_1-\cdots-Q_{10}$ is isomorphic to  $\pi^*(\sfR_S) \cong \sfR_S(2)$. The involutions K, B, and G are induced by a Cremona involution of $\bbP^3$ that leave $Y$ invariant. Let us describe them.

The Kantor involution is  defined by the linear system $|Q|$ of quartics through the first 7 nodes $p_1,\ldots,p_7$ of the symmetroid. This linear system defines a degree 2 rational map $\bbP^3\dasharrow \Sigma\subset \bbP^6$, where $\Sigma $ is a cone over the Veronese surface in $\bbP^5$. We consider the elliptic fibration on the blow-up $\tilde{\bbP}^3$ of 7 points of $\bbP^3$  defined by the net of quadrics through the seven points. It has the negation birational involution defined by the eighth base point. If we take a quartic elliptic curve $E$ through $p_1,\ldots,p_7$ passing through a point $p$, then a quartic surface from the linear system $|Q|$  that contains $p$ intersects $E$ at one more point $p'$. This defines an involution on $E$, $p\mapsto p'$. The set of fixed points is the set of quartics from $|Q|$ that has an eight node. Thus the three remaining nodes of the symmetroid are fixed and this makes the symmetroid invariant.  Note that the birational Kantor involution of $\bbP^3$ is an analog of the Bertini involution of the plane. 

The other two involutions are \emph{dilated} Bertini and Geiser involutions of the plane. They extend these involutions to 
the three dimensional space. 

The classification of finite subgroups of $\Aut(S)$ is far from being complete. We do not even know what is the list of possible automorphism groups of an unnodal Enriques surface. However, we can mention the following result (explained to me by Daniel Allcock).

\begin{theorem}\label{al1} Let $G$ be a finite non-trivial subgroup contained in $W_S(2)$. Then it is a group of order 2, and all such subgroups are conjugate in $W_S$. The quotient $S/G$ is isomorphic to a 4-nodal quartic del Pezzo surface.
\end{theorem}

\begin{proof} We identify $\Num(S)$ with the lattice $\bbE_{10}$.  Suppose $G$ contains an element $\sigma$ of order $2$. Then $V =\bbE_{10}\otimes \bbQ$ splits into the orthogonal direct sum of eigensubspaces $V_+$ and $V_-$ with eigenvalues $1$ and $-1$. 
For any $x= x_++x_-\in \bbE_{10}, x_+\in V_-+,x_-\in V_-$, we have 
$$\sigma(x)\pm x = (x_+-x_-)\pm (x_++x_-)\in 2\bbE_{10}.$$
This implies $2x_{\pm}\in 2\bbE_{10}$, hence $x_{\pm}\in \bbE_{10}$ and  the lattice $\bbE_{10}$ splits into the orthogonal sum of sublattices $V_+\cap \bbE_{10}$ and $V_-\cap \bbE_{10}$. Since $\bbE_{10}$ is unimodular, the sublattices must be unimodular. This gives $V_+\cap \bbE_{10} \cong U$ or $E_8$ and $V_-\cap \bbE_{10} \cong E_8$ or $U$, respectively. Since $\sigma = g_*$ leaves invariant an ample divisor,  we must have $V_+\cap \bbE_{10} \cong U$.
Thus $\sigma = 1_U+(-1)_{E_8}$ and hence $\sigma = g_*$ for some deck transformation of $S\to \calD$.

Suppose $G$ contains an element $\sigma$ of odd order $m$. Then $\sigma^m-1 = (\sigma-1)(1+\sigma+\cdots+\sigma^{m-1}) = 0$, hence, for any $x\not\in 2\bbE_{10}$ which is not $\sigma$-invariant, we have 
$$x+\sigma(x)+\cdots+\sigma^{m-1}(x) \equiv  mx \mod 2\bbE_{10}.$$
Since $m$ is odd, this gives $x\in 2\bbE_{10}$, a contradiction.

Finally, we may assume that $G$ contains an element  of order $2^k, k > 1$. Then it contains an element $\sigma$ of 
order $4$. Let $M = \Ker(\sigma^2+1)\subset \bbE_{10}$. Since $\sigma^2 = -1_{E_8}\oplus 1_U$ for some direct sum decomposition $\bbE_{10}= E_8\oplus U$, we obtain $M\cong E_8$. The equality 
$(\sigma^2+1)(\sigma(x)) = \sigma^3(x)+\sigma(x) = -(\sigma^2+1)(x)$, implies that  $\sigma(M) = M$. Consider $M$ as a module 
over the principal ideal domain $R = \bbZ[t]/(t^2+1)$. Since $M$ has no torsion, it is isomorphic to $R^{\oplus 4}$. This implies that there exists $v,w\in M$ such that $\sigma(v) = w$ and $\sigma(w) = -v.$ However, this obviously contradicts our assumption that $\sigma\in W(\bbE_{10})(2)$.
 \end{proof}
 
Over $\bbC$ one can use the coarse moduli space of Enriques surfaces to show that a general (resp. general nodal) Enriques surface has automorphism group isomorphic to $W(\bbE_{10})(2)$ (resp. $W_{2,4,6}(2)'$). I believe that the same is true in any characteristic  but I cannot prove it (except in the case of  $\bmu_2$-surfaces or general unnodal surface if $p\ne 2$). 

In any case if $p\ne 2$, the image of the  automorphism group $\Aut(S)$ in $W(\bbE_{10})$ is not the whole group. This is because $W(\bbE_{10})$ contains a subgroup isomorphic to $W(E_8)$ and the known information about finite groups of automorphisms of K3-surfaces shows that the order of this group is too large to be realized as an automorphism group of a K3-surface and hence of an Enriques surface.

When the root invariant of an Enriques surface becomes large, the automorphism group may become a finite group. The first example of an Enriques surface with a finite automorphism group isomorphic to $\frakS_4$ belongs to G. Fano \cite{Fano1}. However, I failed to understand Fano's proof. An example of an Enriques surface with automorphism group isomorphic to the dihedral group $D_4$ of order 8 was given in my paper \cite{DolgInv}. At that time I did not know about Fano's example. Later on all complex Enriques surfaces with finite automorphism groups were classified by Nikulin \cite{Nik} (in terms of their root invariant $\sfR_S$ and   by S. Kond\={o} \cite{KondoF} by explicit construction). There  are  seven classes of such surfaces 
with automorphisms groups 
$$D_4,\  \frakS_4,\   2^4\rtimes D_4, 2^2\rtimes (\bbZ/4\bbZ\ltimes \bbZ/5\bbZ),\ \bbZ/2\bbZ\rtimes \frakS_4, \frakS_5,\ \frakS_5.$$
Their Nikulin $R$-invariants  are, respectively,
$$(E_8\oplus A_1,\{0\}),\ (D_9,\{0\}), (D_8\oplus A_1^{\oplus 2},\bbZ/2\bbZ), $$
$$(D_5\oplus D_5, \bbZ/2\bbZ), (E_7\oplus A_2\oplus A_1,\bbZ/2\bbZ),\ (E_6\oplus A_4,\{0\}), (A_9\oplus A_1,\{0\}).$$
Note that Kond\={o}'s classification works in any characteristic $\ne 2$ and  there are more examples in characteristic 2 (see \cite{DolgachevT}).

We refer to the latest works in progress of H. Ito, S. Mukai and H.Ohashi  on the classification of finite groups of automorphisms of complex Enriques  surfaces \cite{Mukai1}, \cite{Mukai2}, \cite{Mukai3}, \cite{Ito}. Note, that, via equivariant 
lifting an Enriques surface to characteristic 0, the classification is the same in all characteristics except when $p = 2$ and $S$ is an $\mu_2$ or an $\balpha_2$-surface (see \cite{DolgachevT}, Theorem 2).  Another remark is that any finite subgroup of $W(\bbE_{10})$ is conjugate to a subgroup of $W(R)$, where $R$ is a negative root lattice corresponding to some subdiagram  of the  Dynkin diagram of the root basis $\alpha_0,\ldots,\alpha_9$ \cite{Bourbaki}, Chapter V, \S 4, Exercise 2. The types of maximal subdiagrams with this property  are 
$$D_9,\ A_1+A_8, \ A_1+A_2+A_6, \ A_4+A_5, \ D_5+A_4,\ E_6+A_3, E_7+A_2, E_8+A_1,A_9.$$  
This implies that the image of a finite subgroup $G$ of $\Aut(S)$ in $W(\bbE_{10})$ is isomorphic to a subgroup of $W(R)$, where $R$ is one of the above root systems.

To conclude our survey let me refer to my earlier surveys of the subject \cite{DolgachevP},\cite{DolgachevP2}. Sadly, many of the problems of the theory discussed in these surveys remain unsolved.

\end{document}